\documentclass[12pt]{amsart}

\usepackage[sc]{mathpazo}
\usepackage[T1]{fontenc}

\usepackage{amsmath, amscd, euscript, stmaryrd, hyperref}
\usepackage{amssymb} 
\usepackage{amsrefs}
\usepackage{amsthm}
\usepackage{mathrsfs}
\usepackage{color}
\usepackage{tikz-cd}
\usepackage[OT2,T1]{fontenc}

\usepackage{sseq}

\newcommand{\im}{\text{im}\hspace{2pt}}

\renewcommand{\scr}[1]{\mathscr{#1}}

\newcommand{\bb}[1]{\mathbb{#1}}

\DeclareMathOperator{\Ext}{\text{Ext}}

\newcommand{\Z}{\bb{Z}}
\newcommand{\Q}{\bb{Q}}

\newcommand{\F}{\bb{F}}

\newcommand{\tBP}[1]{BP\langle #1\rangle}

\newcommand{\wt}{\mathrm{wt}}
\newcommand{\AF}{\mathrm{AF}}

\theoremstyle{plain}
\newtheorem{thm}{Theorem}
\newtheorem{prop}{Proposition}
\newtheorem{lem}{Lemma}

\newtheorem{cor}{Corollary}

\theoremstyle{definition}

\newtheorem{ex}{Example}

\title{A new basis for the complex $K$-theory cooperations algebra}
\author{Dominic Leon Culver}\address{University of Notre Dame}\email{dculver@nd.edu}
\date{\today}
\begin{document}

\maketitle

\begin{abstract}
A classical theorem of Adams, Harris, and Switzer states that the 0th grading of complex $K$-theory cooperations, $KU_0ku$ is isomorphic to the space of numerical polynomials. The space of numerical polynomials has a basis provided by the binomial coefficient polynomials, which gives a basis of $KU_0ku$. 

In this paper, we produce a new $p$-local basis for $KU_0ku_{(p)}$ using the Adams splitting. This basis is established by using well known formulas for the Hazewinkel generators. For $p=2$, we show that this new basis coincides with the classical basis modulo higher Adams filtration.
\end{abstract}
\tableofcontents

\section{Introduction}

The cooperations algebra $KU_*KU$ was originally computed by Adams, Harris, and Switzer in \cite{K1971}. They show that $KU_*KU$ is torsion free, and hence the map
\[
KU_*KU\to KU_*KU\otimes\Q\simeq \Q[u^{\pm 1},v^{\pm 1}]
\]
is monic. They determine the image of this map, described in the following theorem.

\begin{thm}[Adams-Harris-Switzer, \cite{K1971}]
The map 
\[
KU_*KU\to KU_*KU\otimes \Q
\]
gives an isomorphism between $KU_*KU$ and the ring of finite Laurent series $f(u,v)$ which satisfy the following condition: for any nonzero integers $h,k$ we have 
\[
f(h\beta,k\beta)\in \Z[\beta^{\pm 1}, h^{-1}, k^{-1}].
\]
\end{thm}

If we are working with the $2$-local complex $K$-theory spectrum $KU$, then we can rewrite this condition as 
\[
KU_0KU_{(2)}\simeq \{f(w)\in \Q[w^{\pm 1}]\mid f(k)\in \Z_{(2)}\, \text{ for all }k\in \Z_{(2)}^{\times}\}
\]
where $w:= v/u$. Since $KU$ is an even periodic ring spectrum, this determines the entire algebra $KU_*ku$. An elegant proof of this fact using an arithmetic square can be found in \cite{BOSS}. In particular, this method allows one to calculate 
\[
KU_0ku_{(2)} = \{g(w)\in \Q[w]\mid g(k)\in \Z_{(2)}\, \text {for all } k\in \Z_{(2)}^{\times}\}
\]
which is known as the space of \emph{2-local semistable numerical polynomials}. This is related to the space of \emph{2-local numerical polynomials}:
\[
A:=\{h(x)\in \Q[x]\mid h(k)\in \Z_{(2)} \text{ for all }k\in \Z_{(2)}\}
\]
via the following change of coordinates
\[
\Z_{(2)}\to \Z_{(2)}^\times;\, k\mapsto 2k+1
\]
A classical result is that the ring $A$ of numerical polynomials is a free $\Z_{(2)}$-module with basis given by the \emph{binomial coefficient polynomials}
\[
p_n(x):= {x\choose n} = \frac{x(x-1)\cdots (x-n+1)}{n!}.
\]
Via the change of coordinates above, we obtain a basis for $KU_0ku$,
\[
g_n(w) = \frac{(w-1)(w-3)\cdots (w-(2n-1))}{2^nn!}.
\]
At any prime $p$, another basis for $KU_0ku_{(p)}$ is discussed by Baker in \cite{p-adic} and \cite{baker-basis}. In these papers, Baker gives a different basis for $KU_0ku_{(p)}$ where the role of the polynomials $p_n(x)$ are replaced by a sequence of Teichm\"uller characters, and he recovers a recursive formula.

When localizing at an odd prime $p$, $KU$ splits as a wedge of suspensions of the \emph{Johnson-Wilson theory} $E(1)$. The homotopy groups of this spectrum are 
\[
\pi_*(E(1)) = \Z_{(p)}[v_1^{\pm 1}].
\]
The connective cover $ku_{(p)}$ splits as a wedge of suspensions of the \emph{truncated Brown-Peterson spectrum} $\tBP{1}$. The homotopy groups of this spectrum are 
\[
\pi_*(\tBP{1}) = \Z_{(p)}[v_1].
\]
When the prime is 2, then the spectra $E(1)$ and $KU_{(2)}$ are equivalent, as are the spectra $\tBP{1}$ and $ku_{(2)}$. Using the K\"unneth spectral sequence, it is shown in \cite{coop_Adams_summand} that
\begin{equation*}
\begin{split}
E(1)_*\tBP{1} &\simeq E(1)_*\otimes_{BP_*} BP_*BP\otimes_{BP_*}\tBP{1}_* \\
		      &\simeq E(1)_*[t_1, t_2, \ldots]/(\eta_R(v_2),\eta_R(v_3), \ldots) 
\end{split}
\end{equation*}
where the $v_i$ denote the Hazewinkel generators for $BP_*$ and $\eta_R$ denotes the right unit for the Hopf algebroid $(BP_*,BP_*BP)$. The splitting of $KU_{(p)}$ at odd primes $p$ gives a map
\begin{equation}\label{map_phi}
\varphi: E(1)_*\tBP{1}\to KU_*ku_{(p)}
\end{equation}
obtained by including the summand. At the prime 2, this map is an isomorphism. 

In this paper, we use the mod $p$ Adams spectral sequence for the spectrum $E(1)\wedge \tBP{1}$ to determine a basis for $E(1)_0\tBP{1}$ in terms of the generators $t_i$. Using the map $\varphi$, we find what semistable numerical polynomials these basis elements correspond to. More specifically, if we set
\[
\varphi_n:= \varphi\left(v_1^{-\frac{p^n-1}{p-1}}t_n\right)
\]
then we determine an inductive formula determining the $\varphi_n$'s. The basis for $E(1)_0\tBP{1}$ will then be the set of certain monomials on the $\varphi_n$'s. This inductive formula stems from formulas for the right unit, $\eta_R$, on the Hazewinkel generators $v_i$. Strangely, these inductive formulas bear a striking resemblance to those of Baker in \cite{p-adic}. The author does not know how these bases are related.

After determining a basis for $E(1)_0\tBP{1}$ at all primes, we focus on the prime 2, in which case $\varphi$ becomes an isomorphism, giving us a new basis for $KU_0ku$. We compare this new basis with the one provided by the $g_n$'s. In particular, it will be shown that the $g_n$-basis and the one produced here will be the same modulo higher Adams filtration. Our basis has the advantage that it is tightly connected to $BP$-theory and the Steenrod algebra. Moreover, our techniques furnish a basis for $E(1)_0\tBP{1}$ at odd primes, which could not be obtained before by the result of Adams-Switzer-Harris.

\subsection*{Conventions} We will write $\zeta_i$ for the conjugates of the polynomial generators $\xi_i$ in the dual Steenrod algebra. When given a prime $p$, we will write $H_*(-)$ for the functor $H_*(-;\F_p)$. We will write $\Ext_{\scr{A}_*}(M)$ for $\Ext_{\scr{A}_*}(\F_p,M)$ when $M$ is a comodule over the dual Steenrod algebra. We will also write $\Ext_{\scr{E}(1)_*}(M)$ for $\Ext_{\scr{E}(1)_*}(\F_p,M)$ when $M$ is a comodule over the Hopf algebra $\scr{E}(1)_* = E(Q_0,Q_1)_*$. If $X$ is a spectrum, we will write $M_*(X;Q_i)$ for the Margolis homology groups $M_*(H_*X;Q_i)$.

\section{Adams spectral sequence calculation of $E(1)_*\tBP{1}$}

We begin by reviewing the calculation of $ku_*ku_{(2)}$ in terms of the Adams spectral sequence
\begin{equation}\label{ASS_kuku}
\Ext_{\scr{A}_*}(H_*(ku\wedge ku))\implies ku_*ku^{\wedge}_2.
\end{equation}
The details of this calculation can be found in \cite{bluebook}. Recall that 
\[
H_*(ku)\simeq (\scr{A}\sslash \scr{E}(1))_*
\]
where $\scr{E}(1)$ denotes the subalgebra of the Steenrod algebra $\scr{A}$ which is generated by the Milnor primitives $Q_0$ and $Q_1$. Thus a change-of-rings shows that the spectral sequence is of the form
\[
\Ext_{\scr{E}(1)}((\scr{A}\sslash \scr{E}(1))_*)\implies ku_*ku^{\wedge}_2.
\]
An important invariant needed in calculating $\Ext$ over the Hopf algebra $\scr{E}(1)$ is the \emph{Margolis homology}. If $X$ is a module over $\scr{E}(1)$, then as $\scr{E}(1)$ is an exterior algebra on $Q_0$ and $Q_1$, the actions by $Q_i$ square to zero, so we may regard $X$ as a chain complex with differentials $Q_i$. We define the \emph{Margolis homology group with respect to $Q_i$} to be 
\[
M_*(X;Q_i):= \ker Q_i/\im Q_i
\]
i.e., the homology of $X$ with respect to the differential $Q_i$. An easy calculation (cf. \cite{bluebook}) shows that 
\[
M_*(ku;Q_0)\simeq P(\zeta_1^2)
\]
and 
\[
M_*(ku;Q_1)\simeq E(\zeta_1^2,\zeta_2^2, \zeta_3^2, \ldots).
\]
There is a \emph{weight filtration} on $\scr{A}_*$ given by setting
\[
\wt(\zeta_k) = 2^{k-1}
\]
and extending to general monomials by 
\[
\wt(xy) = \wt(x)+\wt(y).
\]
The weight filtration gives an algebraic decomposition
\[
(\scr{A}\sslash \scr{E}(1))_*\simeq \bigoplus_{k=0}^\infty M_1(k)
\]
where the $M_1(k)$ denote the subspaces spanned by monomials whose weight is exactly equal to $2k$. These turn out to be subcomodules and they are the homology of the \emph{integral Brown-Gitler spectra}. The Margolis homology of the subcomodules $M_1(k)$ have an interesting property.
\begin{prop}
The Margolis homology groups of $M_1(k)$ are the subspaces of the Margolis homology of $(\scr{A}\sslash \scr{E}(1))_*$ spanned by the weight $2k$ monomials. In particular 
\[
M_*(M_1(k);Q_0) = \F_2\{\zeta_1^{2k}\}
\]
and if the binary expansion of $k$ is 
\[
k= k_0+k_12+k_22^2+\cdots
\]
then 
\[
M_*(M_1(k);Q_1) = \F_2\{\zeta_1^{2k_0}\zeta_2^{2k_1}\zeta_3^{2k_2}\cdots\}.
\]
In particular, the Margolis homology groups of $M_1(k)$ are one dimensional.
\end{prop}

Adams was able to show in \cite{bluebook} that, since the Margolis homology of the $M_1(k)$ are one dimensional, there is an isomorphism 
\[
M_1(k)^*\simeq \overline{\scr{E}(1)}^{\otimes k-\alpha(k)}\oplus F
\]
where $\overline{\scr{E}(1)}$ denotes the augmentation ideal of $\scr{E}(1)$, $F$ is some free $\scr{E}(1)$-module, and $\alpha(k)$ denotes the number of 1's in the dyadic expansion of $k$. Thus, 
\[
\Ext_{\scr{E}(1)_*}(M_1(k))/tors\simeq \Ext_{\scr{A}_*}(H_*(ku^{\langle k-\alpha(k)\rangle}))
\]
where $ku^{\langle i\rangle}$ denotes the $i$th Adams cover of $ku$. From this it follows that the Adams spectral sequence \eqref{ASS_kuku} collapses at $E_2$. 

The algebra $KU_*ku_{(2)}$ is obtained from $ku_*ku_{(2)}$ by inverting the element $v_1$, thus it is the direct sum of the modules
\[
v_1^{-1}\Ext_{\scr{E}(1)}(M_1(k)).
\]
We will now calculate these $v_1$-inverted $\Ext$-groups. Here is an example of the Adams chart for $v_1^{-1}\Ext(M_1(4))$.

\begin{ex}
We will calculate $v_1^{-1}\Ext_{\scr{E}(1)_*}(M_1(4))$ and find a $\Z_{(2)}$-generator in degree 8. Here is a picture of the Adams chart.
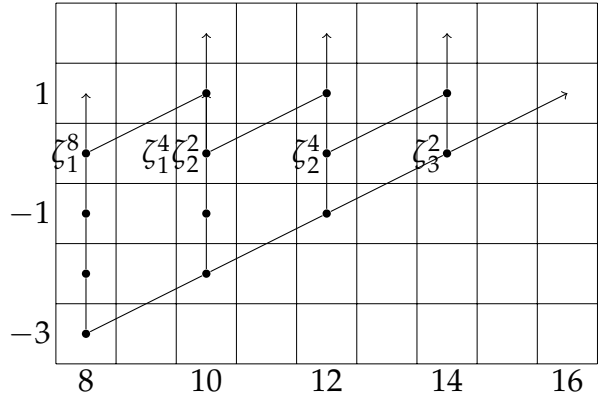
\begin{center}
\begin{sseq}[entrysize=8mm]{8...16}{-3...2}
\ssmoveto{8}{0}
\ssdrop{\bullet}
\ssdroplabel{\zeta_1^8}
\ssvoidarrow 01
\ssline 21
\ssdrop{\bullet}
\ssvoidarrow 01
\ssline{0}{-1}
\ssdrop{\bullet}
\ssdroplabel{\zeta_1^4\zeta_2^2}
\ssvoidarrow 01
\ssline 21
\ssdrop{\bullet}
\ssvoidarrow 01
\ssline{0}{-1}
\ssdrop{\bullet}
\ssdroplabel{\zeta_2^4}
\ssline 21
\ssdrop{\bullet}
\ssvoidarrow 01
\ssline{0}{-1}
\ssdrop{\bullet}
\ssdroplabel{\zeta_3^2}
\ssvoidarrow 21
\ssline{-2}{-1}
\ssdrop{\bullet}
\ssline 01
\ssmove{0}{-1}
\ssline {-2}{-1}
\ssdrop{\bullet}
\ssline 01
\ssdrop{\bullet}
\ssline 01
\ssmove{0}{-2}
\ssline{-2}{-1}
\ssdrop{\bullet}
\ssline01
\ssdrop{\bullet}
\ssline01
\ssdrop{\bullet}
\ssline01
\end{sseq}
\end{center}
This picture is obtained by drawing the Adams chart for $\Ext(M_1(4))$ and then drawing $v_1^{-1}$-towers on each dot on the 0-line. In this example, we see that the relations give $2^3v_1^{-3}\zeta_3^2 = \zeta_1^8$. This shows that the group $v_1^{-1}\Ext^{s,s+8}_{\scr{E}(1)}(M_1(4))$ is generated over $\Z_{(2)}$ by $v_1^{-3}\zeta_1^8$. This also shows that the contribution of $v_1^{-1}\Ext_{\scr{E}(1)}(M_1(4))$ to $KU_0ku$ is the free $\Z_{(2)}$-module generated by $v_1^{-7}\zeta_3^2$.
\end{ex}

\begin{prop}\label{basis_Ext}
Let $k=k_0+k_12+k_22^2+\cdots $ be a natural number, then as a module over $\Z_{(2)}[v_1^{\pm 1}]$, the modules $v_1^{-1}\Ext_{\scr{E}(1)_*}(M_1(k))$ are generated by $v_1^{k-\alpha(k)}\zeta_1^{2k_0}\zeta_2^{2k_1}\cdots$ 
\end{prop}

Recall that in the Adams spectral sequence for $BP_*BP$,  
\[
\Ext_{\scr{E}_*}(P(\zeta_1^2,\zeta_2^2,\zeta_3^2, \ldots ))\implies BP_*BP,
\]
the elements $t_i\in BP_*BP$ are detected by $\zeta_i^2$. Since 
\begin{equation*}
\begin{split}
E(1)_*\tBP{1} &\simeq E(1)_*\otimes_{BP_*} BP_*BP\otimes_{BP_*} BP\langle 1\rangle \\
                      &\simeq E(1)_*[t_1, t_2, \ldots]/(\eta_R(v_2), \eta_R(v_3), \ldots )
\end{split}
\end{equation*}
the elements $\zeta_i^2$ in the Adams spectral sequence for $E(1)_*\tBP{1}$ detect $t_i$. With this notation, we conclude\footnote{Note since $KU_0ku$ has no divisible summands, a set of elements of $KU_0ku_{(2)}$ is a basis if and only if it is a basis of $KU_0ku^{\wedge}_2$.}

\begin{cor}\label{basis_Ext_cor_2}
Let $\varphi_n = v_1^{-2^n+1}t_n$ in $KU_0ku_{(2)}$. The following monomials
\[
\varphi_1^{\epsilon_1}\varphi_2^{\epsilon_2}\cdots
\]
with $\epsilon_j\in \{0,1\}$ forms a basis for the free $\Z_{(2)}$-module $KU_0ku_{(2)}$.
\end{cor}

At an odd prime $p$, the dual Steenrod algebra is given by 
\[
\scr{A}_* = P(\zeta_1, \zeta_2, \ldots)\otimes E(\overline{\tau}_0,\overline{\tau}_1, \ldots)
\]
and the mod $p$ homology of $\tBP{1}$ is given by 
\[
H_*(\tBP{1}) = (\scr{A}\sslash E(Q_0,Q_1))_*
\]
where the $Q_0, Q_1$ are the Milnor primitives. Concretely this algebra is 
\[
(\scr{A}\sslash E(Q_0,Q_1))_* = P(\zeta_1, \zeta_2, \zeta_3, \ldots)\otimes E(\overline{\tau}_2, \overline{\tau}_3, \ldots).
\]
There is a left action of $E(Q_0,Q_1)$ on $(\scr{A}\sslash E(Q_0, Q_1))_*$ given by 
\begin{equation*}
\begin{split}
Q_i(\overline{\tau}_k) &= \zeta_{k-i}^{p^i}\\
Q_i(\zeta_k)&=0
\end{split}
\end{equation*}
for all $k$. This shows that the Margolis homology of $\tBP{1}$ is 
\begin{equation*}
\begin{split}
M_*(\tBP{1};Q_0) &= P(\zeta_1)\\
M_*(\tBP{1};Q_1) &= P(\zeta_1, \zeta_2, \zeta_3, \ldots)/(\xi_1^p, \xi_2^p, \ldots).
\end{split}
\end{equation*}

Similar to the 2-primary case, one can put a \emph{weight filtration} on $(\scr{A}\sslash E(Q_0,Q_1))_*$ by 
\begin{equation*}
\wt(\zeta_k) = \wt(\tau_k) = p^{k}.
\end{equation*}
If we let $M_1(k)$ denote the subcomodule spanned by the monomials of weight exactly $pk$ then we get an algebraic decomposition
\[
(\scr{A}\sslash E(Q_0,Q_1))_*\simeq \bigoplus_{k=0}^\infty M_1(k).
\]
As in the $2$-primary case, the Margolis homology of the subcomodules $M_1(k)$ are both one-dimensional, which from the classification theorem shows that 
\[
M_1(k)^*\simeq \overline{\scr{E}(1)}^{\frac{k-\alpha_p(k)}{p-1}}\oplus F
\]
where $\alpha_p(k)$ is the sum of the digits in the $p$-adic expansion of $k$ and $F$ is a free module. In particular
\[
\Ext_{E(Q_0,Q_1)}(M_1(k))/tors\simeq \Ext_{\scr{A}_*}\left(H_*\left(\tBP{1}^{\left\langle\frac{k-\alpha_p(k)}{p-1}\right\rangle}\right)\right).
\]
From this it follows that the Adams spectral sequence for $\tBP{1}_*\tBP{1}$ collapses at the $E_2$-page. 

Recall that the Adams spectral sequence for $BP_*BP$ at an odd prime is 
\[
\Ext_{E(\overline{\tau}_0,\overline{\tau}_1, \ldots)}(P(\zeta_1, \zeta_2, \ldots))\implies BP_*BP
\]
and in this spectral sequence the $\zeta_k$ detects $t_k\in BP_*BP$. Thus we shall write $t_k$ for $\zeta_k$. Then a proof similar to the proof of Proposition \ref{basis_Ext} shows that 
\begin{prop}
Let the $p$-adic expansion of $k$ be given by $k = k_0+k_1p+k_2p^2+\cdots$. Then over $\Z_{(p)}[v_1^{\pm 1}]$, the module $v_1^{-1}\Ext_{E(\tau_0,\tau_1)}(\tBP{1})$ is generated by 
\[
v_1^{-\frac{k-\alpha_p(k)}{p-1}}t_1^{k_0}t_2^{k_2}t_3^{k_3}\cdots .
\]
\end{prop}

\begin{cor}
Let $\eta_n:= v_1^{-\frac{p^n-1}{p-1}}t_n$. The $\Z_{(p)}$-module $E(1)_0\tBP{1}$ is free with basis given by the monomials
\[
\eta_1^{k_1}\eta_2^{k_2}\cdots 
\]
where each $k_i\in \{0,1, \ldots , p-1\}$.
\end{cor}

\section{Relationship to numerical polynomials}

We will now determine the map
\[
\varphi:E(1)_0BP\langle 1\rangle\to KU_0ku
\]
in terms of numerical polynomials. Recall that the homotopy groups of the \emph{integral} complex $K$-theory spectrum are
\[
\pi_*KU = \Z[v^{\pm 1}]
\]
and thus the rational homotopy groups are
\[
\pi_*(KU_\Q) = \Q[v^{\pm 1}].
\]
Thus we get
\[
\pi_*(KU\wedge KU_\Q) = \Q[v^{\pm 1}, u^{\pm1}]
\]
where we let $u$ denote the Bott element coming from the right hand side $KU$. Similarly, the rational homotopy groups of $KU\wedge ku$ is given by 
\[
\pi_*(KU\wedge ku_\Q) = \Q[v^{\pm 1}, u].
\]
Given a prime $p$, the rational homotopy groups of $E(1)\wedge \tBP{1}$ are given by 
\[
\pi_*(E(1)\wedge \tBP{1}_\Q) = \Q[v_1^{\pm 1}, u_1].
\]
Moreover, at a prime $p$, there is a topological splitting
\[
KU_{(p)}\simeq E(1)\vee \Sigma^2E(1)\vee \cdots \vee \Sigma^{2(p-2)}E(1)
\]
and the inclusion 
\[
E(1)\to KU
\]
is given in homotopy by 
\[
\pi_*E(1)\to \pi_*KU_{(p)}; \,v_1\mapsto v^{p-1}.
\]
Thus the morphism 
\[
\varphi:E(1)\wedge \tBP{1}\to KU\wedge ku_{(p)}
\]
is given in rational homotopy by 
\[
\varphi_\Q:E(1)_*\tBP{1}_\Q\to KU_*ku_\Q\,; \, v_1\mapsto v^{p-1}, \, u_1\mapsto u^{p-1}.
\]
Let $w_1:= u_1/v_1$, then under $\varphi_\Q$, we have that 
\[
w_1\mapsto w^{p-1}.
\]
We will now determine the value of $\varphi$ on the monomials
\[
\varphi_n:= \varphi v_1^{-\frac{p^n-1}{p-1}}t_n.
\]
To do this, we will need the following formula which determines the Hazewinkel generators
\[
p\lambda_n = \sum_{0\leq i <n}\lambda_iv_{n-i}^{p^i}
\]
and the formula for the right unit on $\lambda_n$ 
\[
\eta_R(\lambda_n) = \sum_{0\leq i \leq n} \lambda_i t_{n-i}^{p^i}.
\]
One can find proofs of these formulas in part 2 of \cite{bluebook} and in \cite{Hazewinkelbook}. Here the $\lambda_n$ is the coefficient of $x^{p^n}$ in the logarithm for the universal $p$-typical formal group law. We will show

\begin{thm}\label{psi's}
The semistable polynomials $\varphi_n$ are given recursively by 
\[
\varphi_1= \frac{w^{p-1}-1}{p}
\]
and
\[
\varphi_n= \frac{w^{p^n-1}-p^{n-1}\varphi_{n-1}^p - \cdots  - p\varphi_1^{p^{n-1}}-1}{p^n}.
\]
\end{thm}

We will work out a few examples explicitly and then prove the theorem. Firstly, one has
\[
p\lambda_1 = v_1
\]
and so 
\[
\lambda_1 = \frac{v_1}{p}.
\]
We will write $u_n$ for $\eta_R(v_n)$. This is justified because in $E(1)_*E(1)$, $\eta_R(v_1)$ is $u_1$. Applying $\eta_R$ gives
\[
\eta_R(v_1/p) = \eta_R(\lambda_1) = t_1+\lambda_1
\]
and so 
\[
u_1 = \eta_R(v_1) = pt_1+v_1
\]
Thus
\[
t_1 = \frac{u_1-v_1}{p}
\]
and so 
\[
\varphi_1 = \frac{w^{p-1}-1}{p}.
\]
To get at $\varphi_2$, we need to compute $\eta_R(v_2)$. We have
\[
p\lambda_2 = v_2+ \lambda_1v_1^p
\]
and so 
\[
v_2 = p\lambda_2 - \frac{v_1^{p+1}}{p}.
\]
Applying $\eta_R$ we get
\[
u_2 = p(t_2+\lambda_1t_1^p+\lambda_2) - \frac{u_1^{p+1}}{p}.
\]
Rewriting this, we get
\[
u_2 = pt_2 + v_1t_1^p+v_2+\frac{v_1^{p+1}}{p} - \frac{u_1^{p+1}}{p}.
\]
Tensoring with $BP\langle 1\rangle_*$ produces the following relation in $E(1)_*BP\langle 1\rangle$:
\[
0 = pt_2+v_1t_1^p+\frac{v_1^{p+1}}{p} - \frac{u_1^{p+1}}{p}
\]
and hence
\[
t_2= \frac{u_1^{p+1}-v_1^{p+1}}{p^2} - \frac{v_1t_1^p}{p}.
\]
Multiplying by $v_1^{-p-1}$ gives
\[
v_1^{-p-1}t_2 = \frac{w_1^{p+1}-p(v_1^{-1}t_1)^p-1}{p^2}
\]
which shows that
\[
\varphi_2= \frac{w^{p^2-1}-p\varphi_1^p-1}{p^2}.
\]

We will need the following lemma

\begin{lem}
In $E(1)_*\tBP{1}$ there is the following equality
\[
\lambda_n = \frac{v_1^{\frac{p^n-1}{p-1}}}{p^n}
\]
\end{lem}
\begin{proof}
This follows from the identity 
\[
p\lambda_n = \sum_{0\leq i<n} \lambda_iv_{n-i}^{p^i}
\]
and the fact that in $E(1)_*BP\langle 1\rangle$, $v_k=0$ for $k>1$. Thus $p\lambda_n = \lambda_{n-1}v_1^{p^{n-1}}$. Proceeding inductively gives the identity
\[
\lambda_n = \frac{v_1^{p^{n-1}+p^{n-2}+\cdots + p +1}}{p^n} = \frac{v_1^{\frac{p^n-1}{p-1}}}{p^n}.
\]
\end{proof}

We will prove our theorem from the following proposition.

\begin{prop}
In $E(1)_*BP\langle 1\rangle$, there is the relation
\[
pt_n + \sum_{1\leq i\leq n}\frac{v_1^{\frac{p^i-1}{p-1}}t_{n-i}^{p^i}}{p^{i-1}} = \frac{u_1^{\frac{p^n-1}{p-1}}}{p^{n-1}}.
\]
\end{prop}
\begin{proof}
The formula for the Hazewinkel generators is
\[
p\lambda_n = v_n + \sum_{1\leq i\leq n-1}\lambda_iv_{n-i}^{p^i}
\]
whereby
\[
v_n = p\lambda_n - \sum_{1\leq i\leq n-1}\lambda_iv_{n-i}^{p^i}.
\]
Applying $\eta_R$ then gives
\[
u_n = p\sum_{0\leq i \leq n}\lambda_it_{n-i}^{p^i} - \sum_{1\leq i \leq n-1}\left(\sum_{0\leq j\leq i}\lambda_jt_{i-j}^{p^j}\right)u_{n-i}^{p^i}.
\]
In $E(1)_*BP\langle 1\rangle$, the $u_k$ are zero for $k>1$. So this gives
\[
p\sum_{0\leq i \leq n}\lambda_i t_{n-i}^{p^i} = \sum_{0\leq j\leq n-1}\lambda_jt_{n-1-j}^{p^j}u_1^{p^{n-1}}.
\]
Using the previous lemma, we can rewrite this as 
\begin{align}\label{eq:1}
p\sum_{0\leq i \leq n}\frac{v_1^{\frac{p^i-1}{p-1}}}{p^i} t_{n-i}^{p^i} = \left(\sum_{0\leq j\leq n-1}\frac{v_1^{\frac{p^j-1}{p-1}}}{p^j}t_{n-1-j}^{p^j}\right) u_1^{p^{n-1}}.
\end{align}
We will proceed inductively, the base case being trivial to check. Suppose that we have shown the formula for $n-1$. To complete the induction, it is enough to show that 
\[
\sum_{0\leq j\leq n-1}\frac{v_1^{\frac{p^j-1}{p-1}}}{p^j}t_{n-1-j}^{p^j} = \frac{u_1^{\frac{p^{n-1}-1}{p-1}}}{p^{n-1}}.
\]
Plugging in our inductive formula for $t_{n-1}$: 
\[
t_{n-1} = \frac{u_1^{\frac{p^{n-1}-1}{p-1}}}{p^{n-1}} - \sum_{1\leq k \leq n-1}\frac{v_1^{\frac{p^k-1}{p-1}}}{p^k}t_{n-1-k}^{p^k}
\]
into the right hand side of equation \eqref{eq:1} yields
\[
\frac{u_1^{\frac{p^{n-1}-1}{p-1}}}{p^{n-1}} - \sum_{1\leq k \leq n-1}\frac{v_1^{\frac{p^k-1}{p-1}}}{p^k}t_{n-1-k}^{p^k} + \sum_{1\leq k \leq n-1}\frac{v_1^{\frac{p^j-1}{p-1}}}{p^j}t_{n-1-j}^{p^j} = \frac{u_1^{\frac{p^{n-1}-1}{p-1}}}{p^{n-1}}
\]
which completes the proof.
\end{proof}

We now prove the theorem

\begin{proof}[Proof of Theorem]
By definition, 
\[
\varphi_n:= \varphi\left(v_1^{-\frac{p^n-1}{p-1}}t_n\right).
\]
Observe that
\[
\frac{p^n-1}{p-1} = p^j\frac{p^{n-j}-1}{p-1} + \frac{p^j-1}{p-1}.
\]
This and the proposition then show that 
\[
v_1^{-\frac{p^n-1}{p-1}}t_n = \frac{w_1^{\frac{p^n-1}{p-1}}}{p^n} - \sum_{0<j\leq n}\frac{v_1^{\frac{p^j-1}{p-1}}v_1^{-\frac{p^n-1}{p-1}}}{p^j}t_{n-j}^{p^j} = \frac{w_1^{\frac{p^n-1}{p-1}}}{p^n} - \sum_{0<j\leq n}\frac{(v_1^{-\frac{p^{n-j}-1}{p-1}}t_{n-j})^{p^j}}{p^j}.
\]
Applying $\varphi$ now shows that $\varphi_n$ satisfies the recursive formula, by induction.
\end{proof}

\section{Comparison of the $\varphi_n$ and the $g_n$}

In this section we will let $p=2$, so that $E(1)$ is equivalent to $KU_{(2)}$ and $\tBP{1}$ is equivalent to $ku_{(2)}$. Thus the map $\varphi$ is an isomorphism providing $KU_0ku_{(2)}$ with the basis provided by the $\varphi_n$'s. In this section we compare this basis with the basis provided by the $g_n$'s. In particular we show that the bases are the same modulo higher Adams filtration.  

Recall that in the Adams spectral sequence computing $\pi_*BP$:
\[
\Ext_{\scr{A}_*}(H_*BP)\implies \pi_*BP^{\wedge}_2
\]
the elements $v_i$ have Adams filtration $1$. Also, in the ASS computing $BP_*BP$, 
\[
\Ext_{\scr{A}_*}( H_*(BP\wedge BP))\implies \pi_*(BP\wedge BP)^{\wedge}_2
\]
the elements detecting $t_i$ have Adams filtration $0$. Moreover, the map
\[
\varphi:E(1)_*BP\langle 1\rangle \to KU_*ku_{(2)}
\]
preserve Adams filtration. Therefore, as $\varphi_n$ is the image of $v_1^{-2^n+1}t_n$ under $\varphi$, we can conclude:

\begin{prop}
The Adams filtration of $\varphi_n$ is given by
\[
\AF(\varphi_n) = -(2^n-1).
\]
\end{prop}

The Adams filtration of the semistable numerical polynomial $g_n$ is given by (cf. section 2.3 of \cite{BOSS})
\[
\AF(g_n) = \alpha(n)-2n
\]
where $\alpha(n)$ denotes the number of 1's in the binary expansion of $n$. We will equate the $g_n$ with products of $\varphi_n$ modulo elements of higher Adams filtration. Write out $n$'s binary expansion
\[
n = n_0+n_12+ n_22^2+ \cdots + n_\ell 2^\ell,
\]
then 
\[
\AF(\varphi_1^{n_0}\varphi_2^{n_1}\cdots \varphi_\ell^{n_\ell}) = \sum_{i=0}^\ell n_i(1-2^{i+1}) = \alpha(n)-2n
\]
so $g_n$ and $\varphi_1^{n_0}\varphi_2^{n_1}\cdots\varphi_\ell^{n_\ell}$ have the same Adams filtration. We will prove the following.

\begin{prop}\label{psi_and_g}
Given $n$ and its dyadic expansion
\[
n = n_0+n_12+n_22^2+\cdots 
\]
we have that 
\[
g_n \equiv \varphi_1^{n_0}\varphi_2^{n_1}\cdots \mod \text{higher Adams filtration}.
\]
\end{prop}

To prove this proposition, we will need to prove several lemmas, which is done below.

\begin{lem}\label{psi_mod_AF}
We have 
\[
\varphi_n\equiv \frac{\varphi_1^{2^{n-1}}}{2^{2^{n-1}-1}} \mod \text{higher Adams filtration}.
\]
\end{lem}
\begin{proof}
The map $\varphi$ preserves Adams filtration. Moreover, from Proposition \ref{basis_Ext},  in $v_1^{-1}\Ext(M_1(2^{n-1}))$, there is the relation
\[
2^{2^{n-1}-1}v_1^{-(2^{n-1}-1)}t_n = t_1^{2^{n-1}}.
\]
Multiplying by $v_1^{-2^{n-1}}$ and applying $\varphi$ gives the desired relation. 
\end{proof}

\begin{lem}\label{g_mod_AF}
We have 
\[
g_{n}\equiv \frac{\varphi_1^n}{n!} \mod \text{higher Adams filtration} .
\]
\end{lem}
\begin{proof}
We prove this by induction on $n$. Note that $g_1 = \psi_1$. Suppose that we have shown that 
\[
g_n \equiv \frac{\varphi_1^n}{n!} \mod \text{higher Adams filtration}.
\]
Note that 
\[
g_{n+1} = g_n\cdot \frac{w-(2n+1)}{2(n+1)}.
\]
Even though $\frac{w-(2n+1)}{2(n+1)}$ is not an element of $KU_0ku_{(2)}$, it is an element of $KU_0ku\otimes \Q$. We will show that in $KU_0ku\otimes \Q$, the element $g_{n+1}$ is congruent to $\frac{\varphi_1^{n+1}}{(n+1)!}$ modulo higher Adams filtration in $KU_0ku\otimes \Q$, where Adams filtration is extended to $KU_0ku\otimes \Q$ by setting
\[
\AF\left(\frac{x}{2^i}\right) = \AF(x)-i.
\] 
This will complete the induction process because the map
\[
KU_*ku\to KU_*ku\otimes \Q
\]
preserves Adams filtration and is monic, inducing a monomorphism on associated graded spaces
\[
E^0KU_*ku\to E^0KU_*ku\otimes \Q.
\]
Note that 
\[
\AF(g_{n+1}) = \alpha(n+1)-2n-2
\]
and also that 
\[
\AF\left(\frac{w-(2n+1)}{2(n+1)}\right) = \alpha(n+1)-\alpha(n)-2.
\]
From the formula
\[
\nu_2(n!) = n-\alpha(n)
\]
we find
\[
\nu_2(n+1) = \nu_2((n+1)!)-\nu_2(n!) = 1 - \alpha(n+1)-\alpha(n).
\]
Thus
\[
\AF\left(\frac{w-(2n+1)}{2(n+1)}\right) = \AF\left(\frac{\varphi_1}{n+1}\right) = -1-\nu_2(n+1).
\]
This suggests that these numerical polynomials might be equivalent modulo higher adams filtration. Indeed, 
\[
\frac{w-(2n+1)}{2(n+1)} - \frac{w-1}{2(n+1)} = \frac{2n}{2(n+1)} = \frac{n}{n+1}
\]
and 
\[
\AF\left(\frac{n}{n+1}\right) = \nu_2(n)-\nu_2(n+1) > -1-\nu_2(n+1)
\]
and so, in $E^0(KU_*ku\otimes \Q)$,
\[
\frac{w-(2n+1)}{2(n+1)} \equiv \frac{\varphi_1}{n+1}\mod \text{ higher Adams filtration}
\]
which implies that 
\[
g_{n+1}\equiv \frac{\varphi_1^{n+1}}{(n+1)!}\mod \text{higher Adams filtration}
\]
which completes the induction.
\end{proof}

\begin{cor}\label{g_phi1}
We have the following congruence
\[
g_n\equiv \frac{\varphi_1^n}{2^{n-\alpha(n)}} \mod \text{ higher Adams filtration}.
\]
\end{cor}
\begin{proof}
This is because the 2-adic valuation of $n!$ is 
\[
\nu_2(n!) = n-\alpha(n).
\]
\end{proof}

\begin{lem}\label{g_and_psi_power_2}
There is the following congruence
\[
g_{2^n}\equiv \varphi_{n+1} \mod \text{higher Adams filtration}.
\]
\end{lem}
\begin{proof}
The previous corollary gives that
\[
g_{2^n}\equiv \frac{\varphi_1^{2^n}}{2^{2^n-\alpha(2^n)}} \mod \text{higher Adams filtration}
\]
and by Lemma \ref{psi_mod_AF}
\[
\varphi_{n+1}\equiv \frac{\varphi_1^{2^n}}{{2^{2^n-1}}} \mod \text{higher Adams fitration}.
\]
Since $\alpha(2^n)=1$, this proves the lemma.
\end{proof}

We can now prove Proposition \ref{psi_and_g}

\begin{proof}[Proof of Proposition \ref{psi_and_g}]
First observe that if we take the binary expansion of $n$
\[
n = n_0+n_12+n_22^2+\cdots 
\]
then
\begin{equation} \label{eq:2}
g_n \equiv g_1^{n_0}g_2^{n_1}g_{2^2}^{n_2}\cdots \mod \text{higher Adams filtration}
\end{equation}
Indeed, by Corollary \ref{g_phi1}, we have the congruence
\[
g_n\equiv \frac{\varphi_1^{n}}{n!} \mod \text{ higher Adams filtration}
\]
and 
\[
g_1^{n_0}g_2^{n_1}g_{2^2}^{n_2}\cdots \equiv \left(\frac{\varphi_1}{2^0!}\right)^{n_0}\left(\frac{\varphi_1^2}{2^1!}\right)^{n_1}\left(\frac{\varphi_1^{2^2}}{2^2!}\right)^{n_2}\cdots \mod \text{higher Adams filtration}.
\]
In this last expression, the right hand side is equal to 
\[
\frac{\varphi_1^n}{(2^0!)^{n_0}(2^1!)^{n_1}(2^2!)^{n_2}\cdots }.
\]
So in order to show \eqref{eq:2}, it needs to be shown that 
\[
\nu_2(n!) = \nu_2((2^0!)^{n_0}(2^1!)^{n_1}(2^2!)^{n_2}\cdots).
\]
The right hand is equal to 
\[
\sum_i n_i\nu_2(2^i) = \sum_i n_i2^i-n_i = n - \alpha(n) = \nu_2(n!)
\]
and this proves the congruence \eqref{eq:2}. 

To prove the proposition, apply Corollary \ref{g_and_psi_power_2} to the right hand side of \eqref{eq:2}. This gives
\[
g_n \equiv \varphi_1^{n_0}\varphi_2^{n_1}\varphi_3^{n_2}\cdots 
\]
completing the proof of Proposition \ref{psi_and_g}.
\end{proof}

\bibliographystyle{plain}
\bibliography{basis_cooperations_Adams_summand} 

\end{document}